\pdfoutput=1
\documentclass{amsart}
\usepackage{graphicx, color, overpic}
\usepackage[hidelinks]{hyperref}
\usepackage[nameinlink]{cleveref}

\definecolor{cerulean}{rgb}{0,.48,.65} 
\definecolor{magenta}{rgb}{.5,0,.5} 
\definecolor{dred}{rgb}{.5,0,0} 
\definecolor{green}{rgb}{0,.5,0} 
\definecolor{blue}{rgb}{0,0,1} 
\definecolor{black}{rgb}{0,0,0} 
\definecolor{dgreen}{rgb}{0,.3,0} 
\definecolor{vdred}{rgb}{.3,0,0} 
\definecolor{red}{rgb}{1,0,0} 
\definecolor{salmon}{rgb}{0.98,0.50,0.45} 
\definecolor{gray}{rgb}{.5,.5,.5} 
\definecolor{seagreen}{rgb}{0.13,0.70,0.67} 
\definecolor{chartreuse}{rgb}{0.40,0.80,0.00}
\definecolor{cornflower}{rgb}{0.39,0.58,0.93} 
\definecolor{gold}{rgb}{0.80,0.68,0.00}

\usepackage{nima2}

\DeclareMathOperator{\CAT}{CAT}

\title[Strong shortcuts]{Strong shortcuts, generating sets, \\ and isometric circles in asymptotic cones}

\author{Nima Hoda}
\address{Department of Mathematics, Cornell University \\
  Ithaca, NY 14853, USA} \thanks{The first named author was partially
  supported by an NSERC Postdoctoral Fellowship.}
\email{nima@nimahoda.net}

\author{Timothy Riley}
\address{Department of Mathematics, Cornell University \\
  Ithaca, NY 14853, USA}
\email{tim.riley@math.cornell.edu}

\date{\today}


\begin{document}

 \begin{abstract}
  We show that whether loops can be shortcut in a group's Cayley graph depends on the choice of finite generating set.  Our example is the direct product of two rank-2 free groups and a consequence is that this group has asymptotic cones with isometrically embedded circles that are null-homotopic.
     \end{abstract}

\maketitle


\section{Introduction}

 A metric space is \emph{strongly shortcut} \cite{Hoda:shortcut_graphs:2022, Hoda:shortcut_space} when any sufficiently long loop can be subdivided into two   loops that are shorter by some constant multiplicative factor.  In more detail, suppose $\rho: [0,1] \to X$ with $\rho(0) = \rho(1)$ is a loop in a metric space $X$.  A path $\mu$ in $X$ from $\rho(p)$ to $\rho(q)$, where $0 \leq p \leq q \leq 1$, subdivides $\rho$ into two loops: the concatenation $\rho_1$ of the reverse of $\rho |_{[p,q]}$ with $\mu$, and  the concatenation $\rho_2$ of $\rho |_{[q,1] \cup [0,p]}$ with $\mu$.   We say $X$ is \emph{strongly shortcut} when there exists $L \geq 0$ and $\lambda <1$ such that every  loop $\rho: [0,1] \to X$ of length $\ell \geq L$ can be \emph{subdivided} into two loops $\rho_1$ and $\rho_2$ each of length at most $\lambda \ell$.  If $X$ is a graph, it is strongly shortcut if and only if there exists $L \geq 0$ and $\lambda <1$  such that every combinatorial loop in $X$ can be shortcut into two combinatorial loops in the above manner \cite[Theorem~B]{Hoda:shortcut_space}.  The strong shortcut property is the instance of Gromov's loop subdivision property \cite[\S5.F]{Gromov:asymptotic_invariants:1993} where the loop is required to be subdivided into precisely two sub-loops.   

A \emph{word} $w$ on an alphabet $A$ is a string $a^{\varepsilon_1}_1 \cdots a^{\varepsilon_n}_n$ where each $a_i \in A$ and each $\varepsilon_i \in \{ \pm 1\}$.  We write $|w| =n$.  A \emph{cyclic conjugate} of $w$ is a word $a^{\varepsilon_{i+1}}_{i+1} \cdots a^{\varepsilon_n}_n a^{\varepsilon_{1}}_{1} \cdots a^{\varepsilon_i}_i$ for some $i \in \{1, \ldots, n \}$.

A finitely generated group $G$ is \emph{strongly shortcut} when, for some choice of finite generating set $A$,  its Cayley graph $\textup{Cay}(G, A)$  is strongly shortcut;  or, equivalently, when $G$,  acts properly and cocompactly on a strongly shortcut graph \cite{Hoda:shortcut_graphs:2022}. For a Cayley graph $\textup{Cay}(G, A)$, the strong shortcut property amounts to: there exists $L>0$ and $\lambda <1$ such that if $w$ is a word on $A$ that represents the identity and  $|w| \geq L$, then $w$ has a cyclic conjugate expressible as a concatenation $w_1 w_2$ of words $w_1$ and $w_2$ such that $(1-\lambda)|w| \le |w_1| \leq |w_2|$  and for which there exists a word $\mu$ that equals   $w_1$ in $G$ and satisfies  $|\mu | \leq \lambda |w_1|$.

    The purpose of  this note is to prove that  the choice of finite generating set can matter, answering a question \cite[Section~9, Qu.~3]{chw:2022} of Cashen, Hoda, and Woodhouse.  We prove:

\begin{thm} \label{main}
The Cayley graph $\textup{Cay}(G, A)$  of the product $$G = F_2 \times F_2 = F(a,b) \times F(c,d)$$  of two rank-2 free groups is strongly shortcut when $A = \{a,b,c,d\}$ but not when $A = \{a,b,c,db^{-1}\}$.  
\end{thm}


In \cite{Hoda:shortcut_graphs:2022} Hoda exhibited a   group (namely, $\langle a,t \mid t^{-1}a t =a^2 \rangle$) such that for one finite generating set (namely, $\{ a, t \}$) there is an upper  bound on the lengths of isometrically embedded cycles in the Cayley graph (the \emph{shortcut property}), but for another finite generating set (namely, $\{ a, t, t^2 \}$) there is no such bound.  Our examples of Theorem~\ref{main}  strengthen this  in that for $A = \{a,b,c,d\}$, there exists $\lambda <1$ such that  every combinatorial loop  in  $\textup{Cay}(G, A)$ of length  $\ell > 4$ can be \emph{strongly} shortcut into a pair of loops, each of length at most $\lambda \ell$, but if $A = \{a,b,c,db^{-1}\}$ then there is no upper bound on the lengths of  isometrically embedded cycles.

For $A = \{a,b,c,d\}$, every asymptotic cone of $\textup{Cay}(G, A)$ is a product of  a pair of   $\R$-trees, this product having the $\ell_1$-metric, and so these cones contain no isometrically embedded circles---loops can be shortened within one factor. The asymptotic cones of a group defined with respect to different finite generating sets are all biLipschitz equivalent, so all asymptotic cones of $G = F_2 \times F_2$ are contractible.  Nevertheless, via \cite[Theorem~A]{Hoda:shortcut_space}, Theorem~\ref{main} implies that asymptotic cones of $G$ with $A = \{a,b,c,db^{-1}\}$ have isometrically embedded circles:   

\begin{cor} \label{cor}
$F_2 \times F_2$ has asymptotic cones with isometrically embedded circles that are homotopically
trivial.
\end{cor}

By contrast, it follows from   \cite[Lemma 1.7]{Creutz:normed:2022}  and  \cite[Proposition~4.4]{Hoda:2023} that there are no isometrically embedded circles in the asymptotic cones of finitely generated abelian groups; there are also none in the asymptotic cones of hyperbolic groups because they are $\mathbb{R}$-trees.  Corollary~\ref{cor} shows that relaxing the abelian or hyperbolicity hypothesis in a most elementary manner can give groups with   asymptotic cones that have isometrically embedded circles, despite all being simply connected.  Cashen, Hoda, and Woodhouse \cite{chw:2022} were the first to establish the existence of groups with this property, namely snowflake groups of Brady and Bridson.  Corollary~\ref{cor} provides a more elementary example and answers in the affirmative a question they asked \cite[Section~9, Qu.~6]{chw:2022} as to whether a group with quadratic isoperimetric function can exhibit  this behaviour.

\section{Proof}

We have $G =  F(a,b) \times F(c,d)$.  
That $\textup{Cay}(G, \{a,b,c,d\})$ is strongly shortcut is an instance of the part of \cite[Theorem~C]{Hoda:shortcut_graphs:2022} which states that the $1$-skeleton of a finite dimensional $\CAT(0)$ cube complex has this property.  However, as this is an elementary example, it merits an elementary proof. The following lemma details how loops in $\textup{Cay}(F(a,b), \{a,b\})$ admit shortcuts, and afterwards we will explain how to promote this to the  product. (The words $u= a^ma^{-m} a^{-m} a^{m} b^m b^{-m}$ demonstrate the inequality the lemma presents to be sharp.) 

\begin{lem}\label{shortcuts in free groups}
	If a word $u$ represents $1$ in $F(a,b)$, then it has a cyclic conjugate that can be expressed as a concatenation $u_1 u_2$ of words $u_1$ and $u_2$ that both represent $1$ in $F(a,b)$ and $$\min \{|u_1|, |u_2| \} \geq \lfloor |u|/3 \rfloor.$$ 
\end{lem}

\begin{proof}
 Given a finite tree $T$ with $N$ vertices, there is a vertex $\ast$ so that removing $\ast$ and the incident edges  subdivides $T$ into subtrees $T_1, \ldots, T_k$ each with at most $N/2$ vertices.  This is known as a \emph{centroid decomposition}.   So removing $\ast$ from $T$ and taking the closures $\widehat{T}_1, \ldots \widehat{T}_k$ of the resulting connected components gives subtrees so that $\widehat{T}_i$ is $T_i$ with one additional edge, $\widehat{T}_i$ has at most $(N/2)+1$ vertices, and there are vertices $\ast_i$ in  $\widehat{T}_i$ so that identifying $\ast_1, \ldots, \ast_k$ reconstitutes $T$.

Given a word $u$ that represents $1$ in $F(a,b)$, there is a finite planar tree $T$ with each edge directed and labelled by $a$ or $b$ so that on reads $u$ around $T$ from some base vertex.  The number of edges in $T$ is $|u|/2$, and so the  number of vertices is $N=(|u|/2)+1$.  By centroid decomposition,  there exists a cyclic conjugate  $w_1 \cdots w_k$ of $u$ so that for all $i$,  $w_i$  represents $1$ in $F(a,b)$  and $(|w_i|/2) +1 \leq (N/2) +1 = (((|u|/2)+1)/2) +1$.  (The point is that $w_i$ is read around  $\widehat{T}_i$, which has $(|w_i|/2) +1$ vertices.)  That is,  $|w_i| \leq     (|u|/2) +1$.

If $k =2$, then let $u_1 = w_1$ and $u_2 = w_2$, and the result is straight-forward.   

Assuming then that $k \geq 3$, we have either $|w_1 w_2| \leq |u|/2$ or $|w_3  \cdots w_k| \leq |u|/2$.   In the former case coalesce $w_1 w_2$ and repeat.  In the latter case redefine $w_3$ as  $w_3 \cdots w_k$.    The result is   a cyclic conjugate  $w_1  w_2 w_3$ of $u$ such that $|w_i| \leq     (|u|/2) +1$ and $w_i=1$ in $F(a,b)$ for $i =1, 2, 3$.          

Take $u_1$ be the longest of $w_1$, $w_2$, and $w_3$ and $u_2$ to be the concatenation of the other two in the appropriate order so that $u_1u_2$ is a cyclic conjugate of $u$.  Then $u_1 =u_2=1$ in $F(a,b)$.  Further,   $|u_1| \geq |u|/3$ because it is the maximum of three natural numbers that sum to $|u|$.  And   $|u_2| = |u| - |u_1| \geq   |u| - ((|u|/2) +1) = (|u|/2) -1 \geq \lfloor |u|/3 \rfloor$.  (The case $|u|=3$, where the final inequality fails, does not arise because $u=1$ in $F(a,b)$ implies that $|u|$ is even.)        
\end{proof}

Suppose $w$ is a word on $a,b,c,d$ which represents $1$ in $G$. Without loss of generality, we may assume that the total  number of $a^{\pm 1}$ and $b^{\pm 1}$ letters in $w$ is at least the number of   $c^{\pm 1}$ and $d^{\pm 1}$ letters.  So the word $u$  obtained from $w$ by deleting all $c^{\pm 1}$ and $d^{\pm 1}$ letters satisfies $|u| \geq |w|/2$.   Now $u$ represents $1$ in $F(a,b)$ and, per Lemma~\ref{shortcuts in free groups} applied to $u$, we have  $|u_1|, |u_2| \geq \lfloor |w|/6 \rfloor$.  So for some $i \in \{1,2\}$, some cyclic conjugate of $w$ has a subword  $w_1$ of length at most $|w|/2$ that contains all the letters of $u_i$.  Deleting the letters of this $u_i$ from $w_1$ gives a word that equals $w_1$ in $G$ and shortcuts $w$ so as to establish the strong shortcut property for $\textup{Cay}(G, \{a,b,c,d\})$.

\medskip

We now turn to showing that for $t = db^{-1}$, the Cayley graph $\Gamma = \textup{Cay}(G, \{a,b,c,t\})$ is not
strongly shortcut. With this generating set,  $G$ has the presentation
\begin{equation}\label{pres}
  G = \langle a,b,c,t \mid [a,c],\; [a, tb],\; [b,c],\;
  [b,t]\rangle.
\end{equation}
For all $n \in \N$, the length-$(4n +4)$ word
\begin{equation}\label{cycle}
  w_n = [t^nct^{-n},a] = t^nct^{-n}a t^n c^{-1}t^{-n}a^{-1},
\end{equation}
represents  $1$ in $G$, as can be seen  per the van~Kampen diagram in \Cref{ziggurat} (illustrating the case $n=5$), or by observing that 
$t^n c t^{-n} = (d b^{-1})^n c(d b^{-1})^{-n} =
d^n c d^{-n}$, which commutes with $a$ in $G$.   

We will prove that for all $n \geq 0$, the word $w_n$ defines an isometrically embedded cycle in $\Gamma$, and so cannot be shortcut. The ``ziggurat'' nature of the van~Kampen diagram in \Cref{ziggurat} gives it the property that every edge-path connecting a pair of antipodal vertices on the diagram's perimeter has length (meaning the number of edges comprising the edge-path) greater than or equal to the length of either edge-path around the perimeter connecting them. This is a necessary condition for $w_n$  to define an isometrically embedded cycle in $\Gamma$ and is what led us to Theorem~\ref{main}.   

To prove $w_n$ defines an isometrically embedded cycle in $\Gamma$ we will show that all words of length $2n+2$ that are subwords of cyclic conjugates of $w_n$ are geodesic.  The words in question are for $k \in \{0,1,2,\ldots,n\}$,
$$\begin{aligned} 
u_k & = t^{n-k}c t^{-n}a t^k,  & \qquad  u'_k & = t^{-(n-k)}a t^{n}c^{-1}t^{-k}, \\  
u''_k & =  t^{n-k}c^{-1} t^{-n}a^{-1} t^k,  & \qquad  u'''_k & =      t^{-(n-k)}a^{-1} t^{n}ct^{-k}. 
\end{aligned}$$

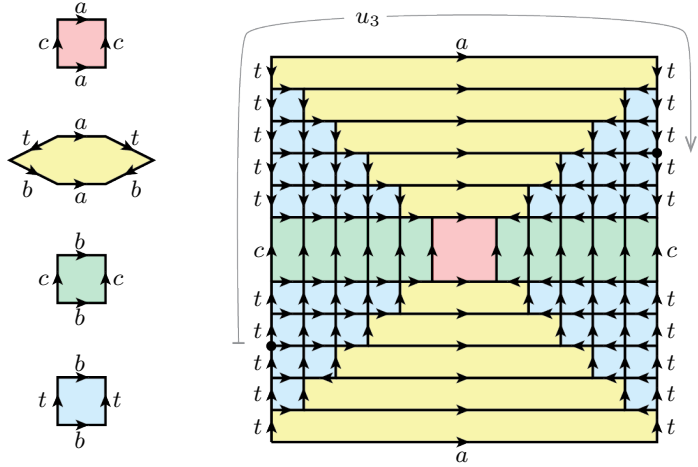
\begin{figure}[ht]
\begin{overpic} 
{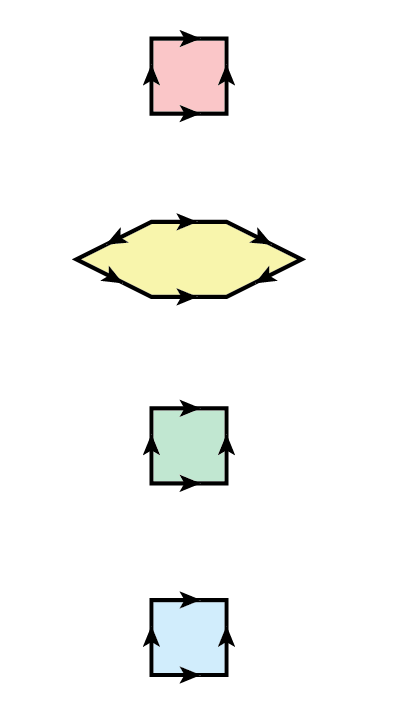}
\put(25, 96){\small{$a$}}
\put(25, 80){\small{$a$}}
\put(17, 88){\small{$c$}}
\put(33.5, 88){\small{$c$}}
\put(25, 70.5){\small{$a$}}
\put(25, 54){\small{$a$}}
\put(13.5, 67){\small{$t$}}
\put(37.5, 67){\small{$t$}}
\put(13.5, 55.5){\small{$b$}}
\put(37.5, 55.5){\small{$b$}}
\put(25, 44){\small{$b$}}
\put(25, 27){\small{$b$}}
\put(17, 36){\small{$c$}}
\put(33.5, 36){\small{$c$}}
\put(25, 17){\small{$b$}}
\put(25, 0){\small{$b$}}
\put(17, 9){\small{$t$}}
\put(33.5, 9){\small{$t$}}
\end{overpic}
\begin{overpic} 
{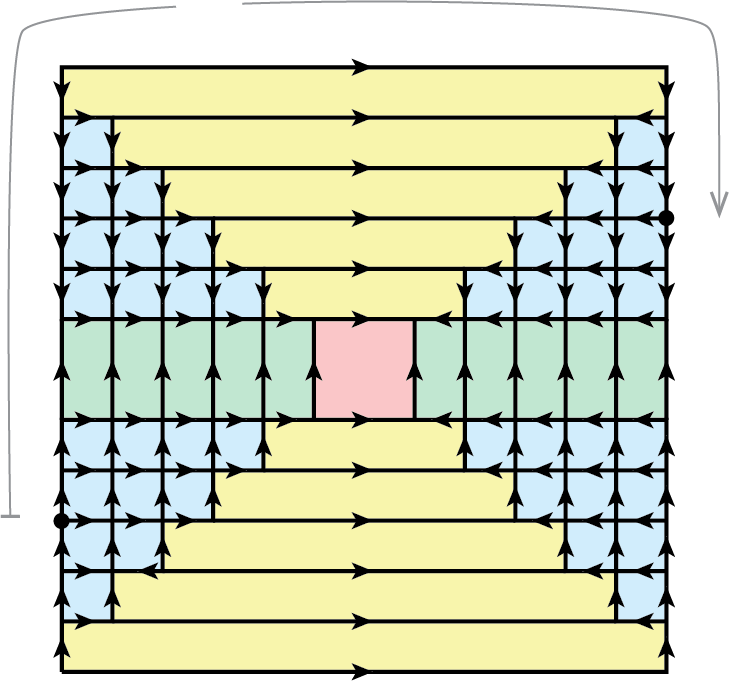}
\put(26.5, 91.5){\small{$u_3$}}
\put(4.5, 3){\small{$t$}}
\put(4.5, 10){\small{$t$}}
\put(4.5, 16.8){\small{$t$}}
\put(4.5, 23.7){\small{$t$}}
\put(4.5, 30.5){\small{$t$}}
\put(4.5, 52){\small{$t$}}
\put(4.5, 59){\small{$t$}}
\put(4.5, 65.9){\small{$t$}}
\put(4.5, 72.8){\small{$t$}}
\put(4.5, 79.7){\small{$t$}}
\put(93.5, 3){\small{$t$}}
\put(93.5, 10){\small{$t$}}
\put(93.5, 16.8){\small{$t$}}
\put(93.5, 23.7){\small{$t$}}
\put(93.5, 30.5){\small{$t$}}
\put(93.5, 52){\small{$t$}}
\put(93.5, 59){\small{$t$}}
\put(93.5, 65.9){\small{$t$}}
\put(93.5, 72.8){\small{$t$}}
\put(93.5, 79.7){\small{$t$}}
\put(4.5, 41){\small{$c$}}
\put(93.5, 41){\small{$c$}}
\put(48, -3){\small{$a$}}
\put(48, 86){\small{$a$}}
\end{overpic}
\caption{Left: the four defining relators of the presentation \eqref{pres} for $G$. Right: a ``ziggurat'' van~Kampen diagram for the word
    $w_5 =   t^5 c t^{-5}a t^5 c^{-1}t^{-5}a^{-1}$.}
\label{ziggurat}
\end{figure}

\begin{claim}
  It suffices to show that $u_k$ is a geodesic word for all $k \in \{0,1,2,\ldots,n\}$.
\end{claim}
\begin{proof}
Let $\phi$ be the automorphism of $G$ that sends $a$ to $a^{-1}$ and fixes $b$, $c$ and $t$. 
Let $\psi$ be the automorphism of $G$ that sends $c$ to $c^{-1}$ and fixes $a$, $b$ and $t$. 
Both induce isomorphisms of $\Gamma$.  Now, $u'_{n-k} = \phi(u_k^{-1})$, $u''_{k} = \phi(\psi(u_k))$, and    $u'''_{n-k} = \psi(u_k^{-1})$, and so if  $u_k$ is a geodesic word, then so are $u'_{n-k}$, $u''_{k}$, and $u'''_{n-k}$.
\end{proof}

Let $D \subset \R^2$ be a van~Kampen diagram with respect to the presentation in
\labelcref{pres}.  Take the union of all open edges labeled $a$ and
open faces having $a$-labeled edges on their boundary.  An
\emph{$a$-corridor} of $D$ is a connected component of this union.
We similarly define \emph{$c$-corridors}.  For example,  there is an $a$-corridor in   \Cref{ziggurat} running vertically through the diagram and a $c$-corridor running horizontally.

Each defining relator in the
presentation in \labelcref{pres} either has no appearance of $a$ or
has exactly two appearances: one of each sign.  Moreover, each relator
in which $a$ appears is a commutator of the form $[a,y]$ where
$y = tb $ or $y = c$.  So any $a$-corridor consists
of a chain of edges and faces.  Moreover, the chain is either linear
in which case the corridor is an open disk with its boundary word
having the form $[a,y_1\cdots y_m]$ or the chain is cyclic in
which case the corridor is an open annulus with its two boundary words
equal and having the form $y_1\cdots y_m$, where each $y_i$ is
either $(tb)^{\pm 1}$ or $c^{\pm 1}$.  The same statements hold
for $c$-corridors with each $y_i$ instead being equal to either
$a^{\pm 1}$ or $b^{\pm 1}$.   For example,   the $a$- and $c$-corridors in \Cref{ziggurat} have boundary words  $[a, (tb)^5 c (tb)^{-5}]$ and $[c, b^5ab^{-5}]$, respectively.

\begin{lem}\label{b_subgroup}
The subgroup $\langle b, t \rangle \leq G$ is an isometrically embedded copy of  $\Z^2 = \langle \, b, t \mid [b,t] \, \rangle$. 
%
\end{lem}

 \begin{proof}
Killing $a$ and $c$ retracts $G$ onto  	$\langle \, b, t \mid [b,t] \, \rangle$.
 \end{proof}

\begin{lem} \label{corridor_backtracks} \label{nonogon} \label{bigon}
  Suppose $D$ is a minimal area van~Kampen diagram for the presentation in
  \labelcref{pres}.  
  \begin{enumerate}
  \renewcommand{\theenumi}{(\alph{enumi})}
\renewcommand{\labelenumi}{\theenumi}
\item  \label{lem1}  If $x$ is the boundary word of an $a$- or $c$-corridor of
  $D$, then $x$ is cyclically reduced or $|x| = 2$. 
\item  \label{lem2} There are  no annular $a$- or $c$-corridors in  $D$.
\item  \label{lem3} The intersection of an $a$-corridor and an
  $c$-corridor in $D$ is either empty or equal to an open face with
  boundary word $[a,c]$.
\end{enumerate}
\end{lem}
\begin{proof}

\ref{lem1}  A cancellable pair of letters in $x$ with $|x| > 2$ would correspond
  to a cancellable pair of faces in the corridor, contradicting
  minimality of $D$.
  
\ref{lem2}  The two boundary words of an annular $a$- or $c$-corridor of $D$ would be
  equal.  Thus we could obtain a new disk diagram of lesser area by
  deleting the corridor and glueing together these two boundary
  components along a label and orientation preserving isomorphism.

\ref{lem3} Since the only defining relator in \labelcref{pres} containing both $a$ and
  $c$ letters is $[a,c]$, the intersection of an $a$-corridor
  with an $c$-corridor is a disjoint union of open faces each of
  which has boundary word $[a,c]$.  If there is more than one open
  face in this intersection then there is a subdiagram $R$ of $D$ that is
  bounded by the union of the two corridors.  The boundary word of an
  innermost such $R$ has no $a$- or $c$-letters and so, by
  \ref{lem2}, has the form $(tb)^{\ell}b^m$.  But
  this word represents the identity in $G$, so by \cref{b_subgroup}, we have
  $\ell = m = 0$.  That is, the region $R$ consists of a single vertex
  of $D$, contained in a cancellable pair of faces each having boundary
  word $[a,c]$.  This contradicts minimality of $D$.
\end{proof}

\begin{claim}\label{ai_instances}
  If $v$ is a geodesic word representing the same element as $u_k$ in
  $G$, then $a$ and $c$ each appear  exactly
  once in $v$ and $a^{-1}$ and $c^{-1}$ do not appear in $v$.
\end{claim}
\begin{proof}
  Let $D$ be a minimal area van~Kampen diagram for the  word
  $u_kv^{-1}$, which represents $1$ in $G$.  Since $a^{-1}$ does not appear in $u_k$, the
  $a$-corridor of $D$ starting at the
  $a$ in $u_k$ must terminate at some $a$ in $v$.
  Thus $a$ appears at least once in $v$.  Any other occurrence of
  $a^{\pm 1}$ in $v$ would correspond to a distinct $a$-corridor.
  Since $u_k$ has no other occurrence of $a^{\pm 1}$, this corridor
  must both start and terminate in $v$.  Since the corridor has boundary word of
  the form $[a,y]$ we can obtain a new disk diagram by deleting the
  corridor and gluing the $y$ and $y^{-1}$ parts of the boundary
  together.  The new disk diagram has boundary word $u_k(v')^{-1}$ for
  the word $v'$ obtained from $v$ by deleting the
  $a^{\pm 1}$ and $a^{\mp 1}$ of the deleted
  corridor, contrary to geodicity of $v$.
  
  The proof for $c$ is the same.  
\end{proof}

We can now prove that $u_k$  is a geodesic word for all $k \in \{0,1,2,\ldots,n\}$. 
Suppose $v$ is a geodesic word representing $u_k$ in $G$.  We will first
consider the case where $c$ occurs before $a$ in $v$.  Let $D$ be
a minimal area van~Kampen diagram for $u_kv^{-1}$.  By \Cref{nonogon} and \Cref{ai_instances}, there is exactly one corridor of each type ($a$
or $c$) in $D$. These two corridors do not cross: were they to cross, they would do so at least twice, but by  \Cref{bigon}\ref{lem3} that does not happen.
Thus, by \Cref{corridor_backtracks}\ref{lem1}, the   $a$-corridor has
boundary word $[a,(tb)^{m}]$, for some $m$, and the
$c$-corridor has boundary word $[c,b^{\ell}]$, for some $\ell$.  Without
loss of generality, a path along the boundary of the $a$-corridor
spelling out $(tb)^{m}$ starts on $u_k$ in $\partial D$ and ends
on $v$ and the same holds for a path along the boundary of the
$c$-corridor spelling out $b^{\ell}$.  See \Cref{two_corridors}.

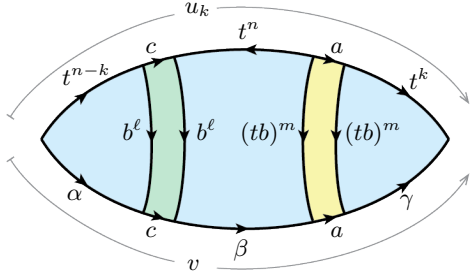
\begin{figure}[ht]
\begin{overpic} 
{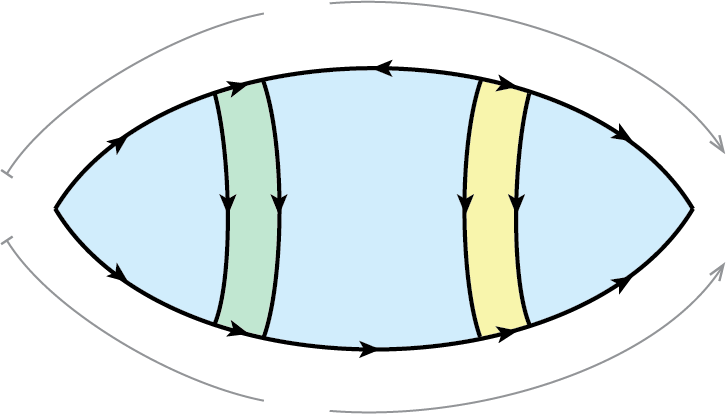}
\put(38.5, 55.5){\small{$u_k$}}
\put(39, 0){\small{$v$}}
\put(70, 7){\small{$a$}}
\put(70, 47){\small{$a$}}
\put(30, 7){\small{$c$}}
\put(30, 47){\small{$c$}}
\put(12, 40){\small{$t^{n-k}$}}
\put(25, 28){\small{$b^{\ell}$}}
\put(41, 28){\small{$b^{\ell}$}}
\put(50, 50){\small{$t^{n}$}}
\put(50.5, 28){\small{$(tb)^{m}$}}
\put(73, 28){\small{$(tb)^{m}$}}
\put(87, 39){\small{$t^k$}}
\put(13, 15){\small{$\alpha$}}
\put(49, 3){\small{$\beta$}}
\put(85, 14){\small{$\gamma$}}
\end{overpic}
\caption{The van~Kampen diagram $D$ for the word $u_kv^{-1}$.}
\label{two_corridors}
\end{figure}

Let $\alpha$, $\beta$, and $\gamma$ be the subwords such that  the word $v$ is the concatenation $\alpha c \beta a \gamma$.  Then $\alpha = t^{n-k}b^{\ell}$, $\beta = b^{-\ell}t^{-n}(tb)^{m}$ and
$\gamma = (tb)^{-m}t^k$ in $G$.  By \Cref{b_subgroup} these are equal in $G$ to
$b^{\ell}t^{n-k}$, $b^{m-\ell}t^{m-n}$ and $b^{-m}t^{k-m}$,
respectively.  Since these are geodesic representatives in
$\langle b, t \mid [b,t] \rangle$, and therefore  (by Lemma~\ref{b_subgroup}) in $G$, it follows that
\begin{align*}
  |u_k| \ge |v| &\ge 2 + (|\ell|+n-k) + (|m - \ell| + |m - n|) + (|m| + |k - m|)\\
                &= 2 + n - k + |n-m| + |m-\ell| + |\ell| + |k-m| + |m|\\
                &\ge 2 + n - k + (n-m) + (m-\ell) + \ell + (k-m) + m\\
                &= 2 + 2n,              
\end{align*}
so that $|u_k| = |v|$ and $u_k$ is geodesic.

The other case is where $a$ occurs before $c$ in $v$. The word  $u_k u''_k$ is a cyclic conjugate of $w_n$, and so $(u''_k)^{-1} = u_k$ in $G$.   So  $(u''_k)^{-1}   =  t^{-k} a  t^{n} c   t^{-(n-k)}$ is a length-$(2n+2)$ word which equals $v$ in $G$.  Comparing them using the same argument as the previous case establishes that    $|u_k| =2n+2$.  This completes our proof.

\bibliographystyle{alpha}
\bibliography{nima}

\begin{thebibliography}{CHW22}

\bibitem[CHW22]{chw:2022}
C.~H. Cashen, N.~Hoda, and D.~J. Woodhouse.
\newblock Asymptotic cones of snowflake groups and the strong shortcut
  property.
\newblock {\em arXiv:2202.11626}, 2022.
\newblock to appear in Algebr. Geom. Topol.

\bibitem[Cre22]{Creutz:normed:2022}
P.~Creutz.
\newblock Rigidity of the {P}u inequality and quadratic isoperimetric constants
  of normed spaces.
\newblock {\em Rev. Mat. Iberoam.}, 38(3):705--729, 2022.

\bibitem[Gro93]{Gromov:asymptotic_invariants:1993}
M.~Gromov.
\newblock Asymptotic invariants of infinite groups.
\newblock In {\em Geometric group theory, {V}ol. 2 ({S}ussex, 1991)}, volume
  182 of {\em London Math. Soc. Lecture Note Ser.}, pages 1--295. Cambridge
  Univ. Press, Cambridge, 1993.

\bibitem[Hod22]{Hoda:shortcut_graphs:2022}
N.~Hoda.
\newblock Shortcut graphs and groups.
\newblock {\em Trans. Amer. Math. Soc.}, 375(4):2417--2458, 2022.

\bibitem[Hod23]{Hoda:2023}
N.~Hoda.
\newblock Crystallographic {H}elly groups.
\newblock {\em Bull. Lond. Math. Soc.}, 55(6):2991--3011, 2023.

\bibitem[Hod24]{Hoda:shortcut_space}
Nima Hoda.
\newblock Strongly shortcut spaces.
\newblock {\em Algebr. Geom. Topol.}, 24(6):3291--3325, 2024.

\end{thebibliography}

\end{document}